\documentclass[reqno,12pt]{amsart}
\usepackage{float,epsfig,here}
\usepackage{enumerate}
\usepackage{amsmath,amssymb}
\newtheorem{lemma}{Lemma}
\newtheorem{assumption}{Assumption}

\newtheorem{theorem}{Theorem}

\allowdisplaybreaks
\newtheorem{proof*}{Proof*}
\newtheorem{definition}{Definition}
\newtheorem{proposition}{Proposition}

\begin{document}

\baselineskip=17pt

\title[]
{Zero and Non-zero Sum Risk-sensitive Semi-Markov Games}
\thanks{}
\author[]{Arnab Bhabak}
\address{Department of Mathematics\\
Indian Institute of Technology Guwahati\\
Guwahati, Assam, India}
\email{bhabak@iitg.ac.in}

\author[]{Subhamay Saha}
\address{Department of Mathematics\\
Indian Institute of Technology Guwahati\\
Guwahati, Assam, India}
\email{saha.subhamay@iitg.ac.in}


\date{}

\begin{abstract} In this article we consider zero and non-zero sum risk-sensitive average criterion games for semi-Markov processes with a finite state space. For the zero-sum case, under suitable assumptions we show that the game has a value. We also establish the existence of a stationary saddle point equilibrium. For the non-zero sum case, under suitable assumptions we establish the existence of a stationary Nash equilibrium.
\vspace{2mm}


\vspace{2mm}


\end{abstract}

\maketitle

\section{Introduction} In this paper we consider both zero and non-zero sum risk-sensitive average criterion games for semi-Markov processes. The state space is assumed to be finite and action spaces are Borel. We also assume that the sojourn times are supported on a fixed compact interval. Under general continuity-compactness assumptions and an additional assumption of irreducibility, we show that the zero-sum game admits a value. We also prescribe a saddle point equilibrium which is given by minimizing and maximizing selectors of a pair of optimality equations. For the non-zero sum game problem, under certain additional assumptions we show the existence of a  Nash equilibrium. In the non-zero sum case the main step involves showing the existence of solution of a coupled system of equations. In the analysis of both the zero-sum and non-zero sum games, risk sensitive games for discrete-time Markov chains serve as an important intermediate step.

                                Stochastic games have found applications in diverse areas like economics, computer networks, evolutionary biology and so on. Generally in stochastic control problems, of which stochastic games are a sub-class, the expectation is minimized or maximized. But the obvious practical issue with this approach is that it does not take into account the controller's attitude towards risk. This motivates the study of risk-sensitive control problems where the expectation of the exponential of the random quantity is considered. Since the pioneering work of Howard and Matheson \cite{Matheson72}, there has been a lot of work on risk-sensitive control of both discrete and continuous time stochastic processes. Risk sensitive games for discrete time Markov chains has been studied by several authors, see for instance \cite{Ghosh14, Rieder17, Daniel19} for zero-sum games and \cite{Ghosh18, Chen19} for non-zero sum games. Risk-sensitive games for continuous-time diffusions has been studied in \cite{Saha20, Pradhan21, Pradhan20}. Similarly, risk-sensitive games for continuous-time Markov chains has been studied in \cite{Pal16, Wei18, Wei19}. The literature on risk-sensitive control of semi-Markov processes is very limited. In \cite{Lian18}, the authors study risk-sensitive control problem for semi-Markov processes on the finite horizon. Risk-sensitive infinite horizon discounted cost problem is considered in \cite{Saha21}. In \cite{Cadena16}, the authors consider the risk-sensitive average cost criterion for semi-Markov processes. But to the best of our knowledge, the present paper is the first work on risk-sensitive semi-Markov games.
 
 The rest of the paper is organised as follows. In section 2, we describe the zero-sum game problem under consideration. In section 3 we introduce the optimality equations and establish its solution. In section 4, we describe the non-zero sum game problem. Section 5 establishes the existence of Nash equilibrium for the non-zero sum game. In section 6 we make some concluding comments.
\section{Zero-Sum Game Model}

The risk-sensitive zero-sum semi-Markov game model that we consider here is given by
\begin{align}\label{Game model}
(S,A,B,\{A(i)\subset A, B(i)\subset B, i\in S\}, C(i,a,b),\{\rho_{(i,a,b)}(\cdot)\}, \{F_{i,a,b}\},[p_{i,j}(a,b)]),
\end{align}
where, \begin{itemize}
	\item $S$ is the state space, which is assumed to be finite and is endowed with the discrete topology.
	\item The Borel spaces $A$ and $B$ are the action sets for player 1 and 2 respectively. And for each $i\in S$, $A(i)\subset A$, $B(i)\subset B$ are Borel subsets denoting the set of all admissible actions in state $i$ for player 1 and 2 respectively.
	\item Define $\mathbb{K}=\{(i,a,b): i\in S, a\in A(i), b\in B(i)\}$ to be the set of admissible state-action pairs. Then $C:\mathbb{K}\to \mathbb{R}$  is the immediate cost function for player 1 and immediate reward for player 2. 
	\item For each $(i,a,b)\in \mathbb{K}$, the mapping $\rho_{(i,a,b)}:[0,\infty)\rightarrow \mathbb{R}$ denotes the running cost function for player 1 and running reward function for player 2. 
	\item $F_{i,a,b}$ is the sojourn time distribution function for both the players in state $i$  under the actions $a$ and $b$. It is assumed that the sojourn times are positive, so that 
	\begin{align}\label{no immediate jump}
	F_{i,a,b}(0)=0,   \hspace{.4cm} (i,a,b)\in \mathbb{K}
	\end{align}
	\item Finally, [$p_{i,j}(a,b)$] is the controlled transition law and satisfies $\sum_{j\in S}p_{i,j}(a,b)=1$ for every $(i,a,b)\in \mathbb{K}$.
\end{itemize}The game evolves in the following manner. At the initial time $t=0$, the process starts at $X_0=i_0\in S$. Suppose player 1 chooses an action $A_0=a_0\in A(i_0)$ and player 2 independently chooses an action $B_0=b_0\in B(i_0)$. As a result player 2 gets an immediate reward $C(i_0,a_0,b_0)$ from player 1. Player 1 also incurs a holding cost at the rate $\rho_{(i_0,a_0,b_0)}$. The process stays in state $i$ for a random amount of time $S_0$ whose distribution function is given by $F_{i_0,a_0,b_0}$ and then jumps to a new state $X_1=i_1$ with probability $p_{i_0,i_1}(a_0,b_0)$. Immediately after the first transition, players 1 and 2 chooses actions $A_1=a_1\in A(i_1)$ and $B_1=b_1\in B(i_1)$. The same sequence of events as described above repeats itself. Let $T_{n}$ to be the time when the nth transition is completed, then 
\begin{align}\label{jump time}
T_{0}=0   \hspace{.5cm}   and   \hspace{.5cm}       T_{n}=\sum_{i=0}^{n-1}S_{i} \hspace{.5cm} n=1,2,... ,
\end{align} 
where $S_{n}, n=0,1,2,\ldots$ denotes the random sojourn times at the nth state. We denote the number of transitions $N_{t}$ in the interval $[0,t]$ by
\begin{align}\label{first hitting time}
N_{t}=\sup\{n\in \mathbb{N}: T_{n}\leq t\},  \hspace{.2cm} t\geq0.
\end{align}
Let $\mathcal{H}_{n}$ be the information available upto time $T_{n}$, i.e., $\mathcal{H}_{0}=X_0$ and  for $n\geq 1$, $\mathcal{H}_{n}=\{X_0,A_0,B_0,S_0,\ldots,X_{n-1}, A_{n-1},B_{n-1},S_{n-1},X_n\}$, where for $n\geq 0$, $X_n$ is the nth state, $A_n $ and $B_n$ are the actions of player 1 and 2 respectively at the nth transition time and $S_n$ is the sojourn time at the nth state. For $n\geq 0$, we also define the admissible history spaces $H_n$ by $H_0=S$ and $H_n=\mathbb{K}\times (0,\infty)\times H_{n-1}$ for $n=1,2,\ldots$. We endow these spaces with the Borel sigma-algebra. 
Now we introduce the concept of policies.
\begin{definition}
	A randomized history dependent policy or simply a policy for player 1 is a sequence $\pi^{1} =\{\pi^{1}_{n}\ : n\geq 0\}$ of stochastic kernels $\pi^{1}_{n}$ on A given $H_{n}$ such that $$\pi^{1}_{n}(A(i_{n})\vert h_{n}) = 1\quad\forall\,\,h_{n}\in H_{n}, n=0,1,....$$
	A randomized history dependent policy for player 2 can be defined analogously.
\end{definition}

Let $\Phi^{1}$ be the set of all stochastic kernels $\phi^1$ on $A$ given $S$ satisfying $\phi^{1}(A(i)\vert i)=1$. A policy $\pi^{1}$ for player 1 is said to be stationary if there exists a stochastic kernel $\phi^{1}\in \Phi^{1}$ such that $\pi_{n}^{1}(.\vert h_{n})=\phi^{1}(.\vert i_{n})$ for all $h_{n}=(i_0,a_0,b_0,s_0,\ldots,i_{n-1},a_{n-1},b_{n-1},s_{n-1},i_n)\in H_{n}$ and $n=0,1,...$. We will identify a stationary policy $\pi^1$ with $\phi^1$. Similarly stationary policies for player 2 can be defined.

For each $m=1,2$, $\Pi_{m}$ and $\Phi^{m}$ represent the set of all randomized history dependent strategies and the set of all stationary strategies for player $m$, respectively. We will have the following assumptions on our model.
\begin{assumption}$\mbox{}$
	\begin{itemize}
		\item[(i)] For each $i\in S$, the set $A(i)$ and $B(i)$ are compact subsets of $A$ and $B$.
		\item [(ii)] For each $i,j\in S$, $(a,b)\rightarrow C(i,a,b)$ and $(a,b)\rightarrow p_{ij}(a,b)$ are continuous in $(a,b)\in A(i)\times B(i)$.
		\item[(iii)] The family $\{F_{i,a,b}\}$ is supported on a compact interval and is weakly continuous, that is, there exists $B>0$ such that
		 \begin{align}\label{compact support of sojourn time}
		F_{i,a,b}(B)=1,   \hspace{.3cm} (i,a,b)\in \mathbb{K},
		\end{align}
		and for each $i\in S$ and $u$ bounded measurable, $(a,b)\rightarrow \int_{0}^{B}u(s)dF_{i,a,b}(s)$ is continuous in $(a,b)\in A(i)\times B(i)$.
		\item[(iv)] For every $i\in S$, the mapping $(a,b,s)\rightarrow \rho_{(i,a,b)}(s)$ is continuous in $(a,b,s)\in A(i)\times B(i)\times [0,B]$. 
		\end{itemize}
		\end{assumption}
Since the spaces $A(i)$ and $B(i)$ are compact and the state space is finite, so it follows by Assumption 1 that,
\begin{align}\label{finiteness of the cost rate}
M_{\rho}:=\sup_{(i,a,b)\in \mathbb{K}, s\in [0,B]}\vert \rho_{(i,a,b)}(s)\vert < \infty.
\end{align}

Given the initial state $X_{0}=i$ and a pair of policies $(\pi^{1}, \pi^{2})$ , the distribution of $\{(X_{n},A_{n},B_{n},\\S_{n})\}$ is uniquely determined by the Tulcea theorem \cite{Marcus93}. We denote such a distribution by $\mathbb{P}^{\pi^{1},\pi^{2}}_{i}$, and $\mathbb{E}^{\pi^{1},\pi^{2}}_{i}$ be the corresponding expectation operator. The following Markov relations are satisfied almost surely under each distribution $\mathbb{P}^{\pi^{1},\pi^{2}}_{i}$: For each $i,j\in S$, $C$ Borel subset of $A$, $D$ Borel subset of $B$ and $n\in \mathbb{N}$,
\begin{align}\label{markov equalities}
&\mathbb{P}^{\pi^{1},\pi^{2}}_{i}[X_{0}=i]=1, \nonumber\\
&\mathbb{P}^{\pi^{1},\pi^{2}}_{i}[A_{n}\in C, B_{n}\in D \vert \mathcal{H}_{n}]= \pi_{n}^{1}(C\vert \mathcal{H}_{n})\pi^{2}_{n}(D\vert \mathcal{H}_{n}), \nonumber\\
&\mathbb{P}^{\pi^{1},\pi^{2}}_{i}[S_{n}\leq t\vert \mathcal{H}_{n},A_{n},B_{n}]=F_{X_{n},A_{n},B_{n}}(t),\nonumber\\
&\mathbb{P}^{\pi^{1},\pi^{2}}_{i}[X_{n+1}=j \vert \mathcal{H}_{n},A_{n},B_{n},S_{n}]=p_{X_{n},j}(A_{n},B_{n}).
\end{align}

Now we describe the evaluation criterion for our game. The total cost incurred by player l and the total reward gained by player 2 up to time $t>0$ is given by:
{\small\begin{align}\label{total cost}
\mathcal{C}_{t}=\sum_{k=0}^{N_{t}-1}\big[C(X_{k},A_{k},B_{k})+\int_{0}^{S_{k}}\rho_{(X_{k},A_{k},B_{k})}(r)dr\big]+C(X_{N_{t}},A_{N_{t}},B_{N_{t}})+\int_{0}^{t-T_{N_{t}}}\rho_{(X_{N_{t}},A_{N_{t}},B_{N_{t}})}(r)dr.
\end{align} }
For risk-sensitivity parameter $\theta>0$ and a policy pair $(\pi^{1},\pi^{2})$ define,
\begin{align}\label{theta sensitive cost}
J_{\theta}(i,\pi^{1},\pi^{2}):= \limsup_{t\rightarrow \infty}\frac{1}{\theta t}\log\left[\mathbb{E}_i^{\pi^1,\pi^2}\left(e^{\theta \mathcal{C}_t}\right)\right].
\end{align}
We further make the following definitions.
\begin{align*}
L(i,\theta)=\sup_{\pi^{2}\in \Pi_{2}}\inf_{\pi^{1}\in \Pi_{1}}J_{\theta}(i,\pi^{1},\pi^{2}),\\
U(i,\theta)=\inf_{\pi^{1}\in \Pi_{1}}\sup_{\pi^{2}\in \Pi_{2}}J_{\theta}(i,\pi^{1},\pi^{2}),
\end{align*}
where $J_{\theta}(i,\pi^{1},\pi^{2})$ is defined in (\ref{theta sensitive cost}). $L(\cdot)$ is called the lower value of the game and $U(\cdot)$ is called the upper value of the game. The value function, if it exists, is denoted by $V(\cdot)$.
\begin{definition}
	If $I(i,\theta)=L(i,\theta)$ for all $i\in S$, then we say that the game has a value. And the common function is referred to as the value of the game. 
\end{definition}
Here player 1 is interested in minimizing $J_{\theta}(i,\pi^{1},\pi^{2})$ over $\pi^{1}\in \Pi_{1}$ for each $\pi^{2}\in \Pi_{2}$, and
player 2 wants to maximize $J_{\theta}(i,\pi^{1},\pi^{2})$ over $\pi^{2}\in \Pi_{2}$ for each $\pi^{1}\in \Pi_{1}$. This motivates the following definition.
\begin{definition}Suppose that the value of the game exists.
	A policy $\pi^{*^{1}}\in \Pi_{1}$ is said to be optimal for player 1, if for any $i\in S$,
	\begin{align*}
	V(i,\theta)=\sup_{\pi^{2}\in \Pi_{2}}J_{\theta}(i,\pi^{*^1},\pi^{2}), \hspace{.2cm} \forall i\in S.
	\end{align*}
	Similarly, for player 2 a policy $\pi^{*^{2}}\in \Pi_{2}$ is said to be optimal, if for any $i\in S$,
	\begin{align*}
	V(i,\theta)=\inf_{\pi^{1}\in \Pi_{1}}J_{\theta}(i,\pi^{1},\pi^{*^2}), \hspace{.2cm} \forall i\in S.
	\end{align*}
	If $\pi^{*^{m}}\in \Pi_{m}$ is optimal for player $m (m=1,2)$, then $(\pi^{*^{1}},\pi^{*^{2}})\in \Pi_{1}\times \Pi_{2}$ is called a saddle point equilibrium.
\end{definition}
\section{Analysis of Zero-Sum Game}
For $i\in S$, let $\mathcal{P}(A(i))$ and $\mathcal{P}(B(i))$ denote the set of all probability measures on $A(i)$ and $B(i)$ respectively. The analysis of the zero-sum game crucially depends on the following equation.
\newpage
\begin{align}\label{optimal equation}
e^{\theta h(i)}&=\sup_{\varphi\in \mathcal{P}(B(i))}\inf_{\psi\in \mathcal{P}(A(i))}\big[\int_{A(i)}&\int_{B(i)}\psi(da) \varphi(db)  e^{\theta C(i,a,b)}\int_{0}^{B} e^{\theta[\int_{0}^{s}\rho_{(i,a,b)}(t)dt-gs]}dF_{i,a,b}(s)\nonumber\\&\times \sum_{j\in S}e^{\theta h(j)}p_{ij}(a,b) \big],  \hspace{.5cm} i\in S.
\end{align}
where $g$ is a real number and $h(.)$ is a real function defined on the state space $S$. Using Assumption 1 and Fan's minimax theorem \cite{Fan52}, equation (\ref{optimal equation}) can also be written as:
\begin{align}\label{optimal equation1}
e^{\theta h(i)}&=\inf_{\psi\in \mathcal{P}(A(i))}\sup_{\varphi\in \mathcal{P}(B(i))}\big[\int_{A(i)}&\int_{B(i)}\psi(da) \varphi(db)  e^{\theta C(i,a,b)}\int_{0}^{B} e^{\theta[\int_{0}^{s}\rho_{(i,a,b)}(t)dt-gs]}dF_{i,a,b}(s)\nonumber\\&\times \sum_{j\in S}e^{\theta h(j)}p_{ij}(a,b) \big],  \hspace{.5cm} i\in S.
\end{align}
The importance of the above equations is illustrated by the next theorem.
\begin{theorem}\label{verification theorem}
	Suppose that equation \eqref{optimal equation} and hence equation \eqref{optimal equation1} is satisfied by a pair $(g,h(.))$. Under Assumption 1, it follows that the game has a value and is given by $g=V(i,\theta)$. Further if $\phi^{*1}\in \Phi^1$ is the outer minimising selector of the right hand side of \eqref{optimal equation1} and if $\phi^{*2}\in \Phi^2$ is the outer maximising selector of the right hand side of \eqref{optimal equation}, then $(\phi^{*1},\phi^{*2})$ is a saddle point equilibrium.
\end{theorem}
In order to prove Theorem \ref{verification theorem}, we need the following auxiliary lemma.
\begin{lemma}\label{technical1}
	Suppose Assumption 1 holds. Then the following holds:
	\begin{itemize}
		\item[(i)] Given $\alpha \in (0,1)$, there exists an integer $r_{\alpha}>0$ such that, for every $(i,a,b)\in \mathbb{K}$, the inequality $\int_{0}^{B}e^{-r s}dF_{i,a,b}(s)\leq \alpha$ holds for every $r \geq r_{\alpha}$.
		\item[(ii)] For each $\alpha \in (0,1)$, $t\geq 0$ and $n\in \mathbb{N}$, $\mathbb{P}_{i}^{\pi^{1},\pi^{2}}[N_{t}\geq n]\leq \alpha^{n}e^{r_{\alpha} t}$ for all $i\in S$ and $(\pi^1,\pi^2)\in \Pi^1\times \Pi^2$, where $r_{\alpha}$ is as in part (ii). Thus,
		\begin{align}\label{finiteness}
		\mathbb{P}_{i}^{\pi^{1},\pi^{2}}[N_{t}< \infty]=1.
		\end{align}
	\end{itemize}
\end{lemma}
\begin{proof}
The proof is a simple generalization of Lemma 4.1 in \cite{Cadena16}.
\end{proof}
\begin{proposition}
	Let $(g,h(.))$ be a solution of equation \eqref{optimal equation} and hence of equation \eqref{optimal equation1} . Under Assumption 1, the following are true.\\
	For each $i\in S$, $(\pi^{1},\pi^{2})\in \Pi_{1}\times \Pi_{2}$ and $t>0$:
	\begin{align}\label{greater than inequality}
	 e^{\theta h(i)}\geq \mathbb{E}_{i}^{\phi^{*1},\pi^{2}}\big[e^{\theta[\sum_{k=0}^{N_{t}}(C(X_{k},A_{k},B_{k})+\int_{0}^{S_{k}}\rho_{(X_{k},A_{k},B_{k})}(s)ds)-gT_{N_{t}+1}+h(X_{{N_{t}+1}})]}\big],
	\end{align}
	and also we have,
	\begin{align}\label{less than inequality}
	e^{\theta h(i)}\leq \mathbb{E}_{i}^{\pi^{1},\phi^{*2}}\big[e^{\theta[\sum_{k=0}^{N_{t}}(C(X_{k},A_{k},B_{k})+\int_{0}^{S_{k}}\rho_{(X_{k},A_{k},B_{k})}(s)ds)-gT_{N_{t}+1}+h(X_{{N_{t}+1}})]}\big],
	\end{align}
	where $\phi^{*1}$ and $\phi^{*2}$ are as in Theorem \ref{verification theorem}.
\end{proposition}
\begin{proof} From \eqref{optimal equation1} we have for any $\varphi \in \mathcal{P}(B(i))$
\begin{align*}e^{\theta h(i)}\geq \big[\int_{A(i)}&\int_{B(i)}\phi^{*1}(da|i) \varphi(db)  e^{\theta C(i,a,b)}\int_{0}^{B} e^{\theta[\int_{0}^{s}\rho_{(i,a,b)}(t)dt-gs]}dF_{i,a,b}(s)\nonumber\\&\times \sum_{j\in S}e^{\theta h(j)}p_{ij}(a,b) \big],  \hspace{.5cm} i\in S.
\end{align*}
Thus for any $\pi^2\in \Pi_2$ we have,
\begin{align}\label{other optimality eqn with expectation}
e^{\theta h(i)}\geq\mathbb{E}_{i}^{\phi^{*1},\pi^{2}}\big[e^{\theta[C(X_{0},A_{0},B_{0})+\int_{0}^{S_{0}}\rho_{(X_{0},A_{0},B_{0})}(t)dt-gS_{0}+h(X_{1})]}\big],  \hspace{.5cm} i\in S.
\end{align}
More generally, via equations \eqref{markov equalities} it follows that for every $n\in \mathbb{N}$,
\begin{align}\label{Markov step}
e^{\theta h(X_{n})}\geq\mathbb{E}_{i}^{\phi^{*1},\pi^{2}}\big[e^{\theta[C(X_{n},A_{n},B_{n})+\int_{0}^{S_{n}}\rho_{(X_{n},A_{n},B_{n})}(t)dt-gS_{n}+h(X_{n+1})]}\vert \mathcal{H}_n\big],  \hspace{.5cm} i\in S.
\end{align}
We prove by induction  that for every non-negative integer n,
\begin{align}\label{induction step}
e^{\theta h(i)}\geq &\mathbb{E}_{i}^{\phi^{*1},\pi^{2}}\big[e^{\theta [\sum_{k=0}^{N_{t}}(C(X_{k},A_{k},B_{k})+\int_{0}^{S_{k}}\rho_{(X_{k},A_{k},B_{k})}(t)dt)-gT_{N_{t}+1}+h(X_{N_{t}+1})]}1_{[N_{t}\leq n]}\big]\nonumber\\&+\mathbb{E}_{i}^{\phi^{*1},\pi^{2}}\big[e^{\theta[\sum_{k=0}^{n} (C(X_{k},A_{k},B_{k})+\int_{0}^{S_{k}}\rho_{(X_{k},A_{k},B_{k})}(t)dt)-gT_{n+1}+h(X_{n+1})]}1_{[N_{t}>n]}\big]
\end{align}
To show this, from \eqref{other optimality eqn with expectation} we get,
\begin{align*}
e^{\theta h(i)}&\geq \mathbb{E}_{i}^{\phi^{*1},\pi^{2}}\big[e^{\theta[C(X_{0},A_{0},B_{0})+\int_{0}^{S_{0}}\rho_{(X_{0},A_{0},B_{0})}(t)dt-gS_{0}+h(X_{1})]}\big]\\
&=\mathbb{E}_{i}^{\phi^{*1},\pi^{2}}\big[e^{\theta[C(X_{0},A_{0},B_{0})+\int_{0}^{S_{0}}\rho_{(X_{0},A_{0},B_{0})}(t)dt-gS_{0}+h(X_{1})]}1_{[N_{t}=0]}\big]\\&+\mathbb{E}_{i}^{\phi^{*^1},\pi^{2}}\big[e^{\theta[C(X_{0},A_{0},B_{0})+\int_{0}^{S_{0}}\rho_{(X_{0},A_{0},B_{0})}(t)dt-gS_{0}+h(X_{1})]}1_{[N_{t}>0]}\big];
\end{align*}
since $T_{1}=S_{0}$, hence we have the basis step for $n=0$. Now suppose that (\ref{induction step}) is true for a non-negative integer $n$. Then we have
\begin{align*}
&e^{\theta[\sum_{k=0}^{n}( C(X_{k},A_{k},B_{k})+\int_{0}^{S_{k}}\rho_{(X_{k},A_{k},B_{k})}(t)dt)-gT_{n+1}+h(X_{n+1})]}1_{[N_{t}> n]}\nonumber\\&=e^{\theta[\sum_{k=0}^{n} (C(X_{k},A_{k},B_{k})+\int_{0}^{S_{k}}\rho_{(X_{k},A_{k},B_{k})}(t)dt)-gT_{n+1}]}1_{[N_{t}\geq n+1]}e^{\theta h(X_{{n+1}})}\\
&\geq e^{\theta[\sum_{k=0}^{n}( C(X_{k},A_{k},B_{k})+\int_{0}^{S_{k}}\rho_{(X_{k},A_{k},B_{k})}(t)dt)-gT_{n+1}]}1_{[N_{t}\geq n+1]}\\&\times \mathbb{E}_{i}^{\phi^{*1},\pi^{2}}\big[e^{\theta[C(X_{n+1},A_{n+1},B_{n+1})+\int_{0}^{S_{n+1}}\rho_{(X_{n+1},A_{n+1},B_{n+1})}(t)dt-gS_{n+1}+h(X_{n+2})]}\vert \mathcal{H}_{n+1}\big]\\
&=\mathbb{E}_{i}^{\phi^{*1},\pi^{2}}\big[e^{\theta[ \sum _{k=0}^{n+1}(C(X_{k},A_{k},B_{k})+\int_{0}^{S_{k}}\rho_{(X{k},A_{k},B_{k})}(t)dt)-g[S_{n+1}+T_{n+1}]+h(X_{n+2})]}\times 1_{[N_{t}\geq n+1]}\vert \mathcal{H}_{n+1}\big]
\end{align*}
where (\ref{Markov step}) was used to deduce the first inequality, whereas the fact that the random variables $1_{[N_{t}\geq n+1]}$ and $\sum _{k=0}^{n}(C(X_{k},A_{k},B_{k})+\int_{0}^{S_{k}}\rho_{(X_{k},A_{k},B_{k})}(t)dt)-gT_{n+1}+h(X_{n+1})$ are $\sigma(\mathcal{H}_{n+1})$-measurable was used in the last step. Since $T_{n+2}=T_{n+1}+S_{n+1}$, by (\ref{jump time}) it follows that
\begin{align*}
&\mathbb{E}_{i}^{\phi^{*1},\pi^{2}}\big[e^{\theta[\sum_{k=0}^{n}(C(X_{k},A_{k},B_{k})+\int_{0}^{S_{k}}\rho_{(X_{k},A_{k},B_{k})}(t)dt)-gT_{n+1}+h(X_{n+1})]}1_{[N_{t}>n]}\big]\\
&\geq \mathbb{E}_{i}^{\phi^{*1},\pi^{2}}\big[e^{\theta[\sum_{k=0}^{n+1}(C(X_{k},A_{k},B_{k})+\int_{0}^{S_{k}}\rho_{(X_{k},A_{k},B_{k})}(t)dt)-gT_{n+2}+h(X_{n+2})]}1_{[N_{t}\geq n+1]}\big]\\
&= \mathbb{E}_{i}^{\phi^{*1},\pi^{2}}\big[e^{\theta[\sum_{k=0}^{n+1}(C(X_{k},A_{k},B_{k})+\int_{0}^{S_{k}}\rho_{(X_{k},A_{k},B_{k})}(t)dt)-gT_{n+2}+h(X_{n+2})]}1_{[N_{t}= n+1]}\big]\\&+\mathbb{E}_{i}^{\phi^{*1},\pi^{2}}\big[e^{\theta[\sum_{k=0}^{n+1}(C(X_{k},A_{k},B_{k})+\int_{0}^{S_{k}}\rho_{(X_{k},A_{k},B_{k})}(t)dt)-gT_{n+2}+h(X_{n+2})]}1_{[N_{t}>n+1]}\big]\\
&= \mathbb{E}_{i}^{\phi^{*1},\pi^{2}}\big[e^{\theta[\sum_{k=0}^{N_{t}}(C(X_{k},A_{k},B_{k})+\int_{0}^{S_{k}}\rho_{(X_{k},A_{k},B_{k})}(t)dt)-gT_{N_{t}+1}+h(X_{N_{t}+1})]}1_{[N_{t}= n+1]}\big]\\&+\mathbb{E}_{i}^{\phi^{*1},\pi^{2}}\big[e^{\theta[\sum_{k=0}^{n+1}(C(X_{k},A_{k},B_{k})+\int_{0}^{S_{k}}\rho_{(X_{k},A_{k},B_{k})}(t)dt)-gT_{n+2}+h(X_{n+2})]}1_{[N_{t}>n+1]}\big].
\end{align*}
so, together with the induction hypothesis it follows that (\ref{induction step}) is also valid for $n+1$. Thus the induction argument is complete. Then Monotone convergence theorem, together with (\ref{finiteness}) gives,
\begin{align}\label{limit}
&\lim_{n\rightarrow \infty}\mathbb{E}_{i}^{\phi^{*1},\pi^{2}}\big[e^{\theta[\sum_{k=0}^{N_{t}}(C(X_{k},A_{k},B_{k})+\int_{0}^{S_{k}}\rho_{(X_{k},A_{k},B_{k})}(t)dt)-gT_{N_{t}+1}+h(X_{N_{t}+1})]}1_{[N_{t}\leq n]}\big]\nonumber\\
&=\mathbb{E}_{i}^{\phi^{*1},\pi^{2}}\big[e^{\theta[\sum_{k=0}^{N_{t}}(C(X_{k},A_{k},B_{k})+\int_{0}^{S_{k}}\rho_{(X_{k},A_{k},B_{k})}(t)dt)-gT_{N_{t}+1}+h(X_{N_{t}+1})]}\big].
\end{align}
Now using Assumption 1 and Lemma \ref{technical1} we get that
\begin{align*}
&\mathbb{E}_{i}^{\phi^{*1},\pi^{2}}\big[e^{\theta[\sum_{k=0}^{n}(C(X_{k},A_{k},B_{k})+\int_{0}^{S_{k}}\rho_{(X_{k},A_{k},B_{k})}(t)dt)-gT_{n+1}+h(X_{n+1})]}1_{[N_{t}>n]}\big]
\rightarrow 0  \hspace{.3cm} as \hspace{.1cm} n\rightarrow \infty.
\end{align*}
Now taking $n\rightarrow \infty$ on both the sides of (\ref{induction step}) and using the last convergence and (\ref{limit}) we get the desired inequality \eqref{greater than inequality}.\\ The other inequality \eqref{less than inequality} also follows analogously starting from \eqref{optimal equation}.
\end{proof}
\textbf{Proof of Theorem \ref{verification theorem}}
We have, $T_{N_{t}}\leq t< T_{N_{t}+1}=T_{N_{t}}+S_{N_{t}}$, for every $t>0$, and thus
\begin{align}\label{jump and sojourn time}
0\leq t-T_{N_{t}}\leq S_{N_{t}}\leq B \hspace{.2cm}  and \hspace{.2cm}  T_{N_{t}+1}-t\leq S_{N_{t}}\leq B.
\end{align}
Now from (\ref{total cost}) we have
\begin{align*}\label{total cost rearrangment}
&\sum_{k=0}^{N_{t}}\big[C(X_{k},A_{k},B_{k})+\int_{0}^{S_{k}}\rho_{(X_{k},A_{k},B_{k})}(r)dr\big]-gT_{N_{t}+1}\nonumber\\&=(\mathcal{C}_{t}-tg)+\int^{S_{N_{t}}}_{t-T_{N_{t}}}\rho_{(X_{N_{t}},A_{N_{t}},B_{N_{t}})}(r)dr-(T_{N_{t}+1}-t)g.
\end{align*}
and together with the equality (\ref{finiteness of the cost rate}) and (\ref{jump and sojourn time}) it follows that
\begin{align}
\bigg|\sum_{k=0}^{N_{t}}\big[C(X_{k},A_{k},B_{k})+\int_{0}^{S_{k}}\rho_{(X_{k},A_{k},B_{k})}(r)dr\big]-gT_{N_{t}+1}-(\mathcal{C}_{t}-tg)\bigg|\leq B(M_{\rho}+\vert g \vert).
\end{align}
Using (\ref{less than inequality}) we get that $e^{-2\theta \vert\vert h\vert\vert}\leq \mathbb{E}_{i}^{\pi^{1},\phi^{*2}}\big[e^{ \theta\big[\sum_{k=0}^{N_{t}}(C(X_{k},A_{k},B_{k})+\int_{0}^{S_{k}}\rho_{(X_{k},A_{k},B_{k})}(r)dr)\big]-gT_{N_{t}+1}}\big]$. Using (\ref{total cost rearrangment}), we have $$e^{-2\theta \vert\vert h\vert\vert}\leq \mathbb{E}_{i}^{\pi^{1},\phi^{*2}}\big[e^{\theta [\mathcal{C}_{t}-tg+B(M_{\rho}+\vert g\vert)]}\big],$$ so that $e^{-2\theta \vert\vert h\vert\vert -\theta B(M_{\rho}+\vert g\vert)+\theta tg}\leq \mathbb{E}_{i}^{\pi^{1},\phi^{*2}}[e^{\theta \mathcal{C}_{t}}]$.
Taking logarithm on both sides, dividing by $\theta t$ and then taking limit $t\to \infty$ we get,
$$g \leq J_{\theta}(i,\pi^1,\phi^{*2}),\quad\forall i\in S.$$
For the other inequality consider inequality (\ref{greater than inequality}). Then proceeding similarly as above we have the following inequality, $$e^{2\theta \vert\vert h\vert\vert+\theta B(M_{\rho}+\vert g\vert)+\theta tg}\geq \mathbb{E}_{i}^{\phi^{*1},\pi^{2}}\big[e^{\theta \mathcal{C}_{t}}\big].$$ Again taking logarithm on both sides, dividing by $\theta t$ and then taking limit $t\to \infty$ we get,
$$g \geq J_{\theta}(i,\phi^{*1},\pi^2),\quad\forall i\in S.$$ Since $(\pi^1,\pi^2)$ was arbitrary, we get
\begin{align*}
g \leq  \displaystyle\inf_{\pi^1\in \Pi_1}J_{\theta}(i,\pi^1,\phi^{*2})\leq L(i,\theta)\leq U(i,\theta)\leq \displaystyle\sup_{\pi^2\in \Pi_2}J_{\theta}(i,\phi^{*1},\pi^2)\leq g.
\end{align*} Hence we have the desired conclusions.

In view of Theorem \ref{verification theorem}, in order to establish the existence of the value of the game and saddle point equilibrium, it sufficies to show the existence of solution of the optimality equation \eqref{optimal equation}. For that we impose one more assumption on our model.
\begin{assumption}
	Under each stationary policy, the embedded discrete-time Markov chain $\{X_{n}\}$ is irreducible.
\end{assumption}
In order to establish the existence of solution of \eqref{optimal equation},  we first consider risk-sensitive average criterion game problem for the discrete time  process $\{X_{n}\}$. For that we consider policies $(\pi^1,\pi^2)$, where for each positive integer $n$, the kernels $(\pi_{n}^{1},\pi_{n}^{2})$ depends only on $X_{0},A_{0},B_{0},X_{1},....,X_{n-1},A_{n-1},B_{n-1},X_{n}$. Given a bounded continuous function $D$ on $\mathbb{K}$, define the discrete-time average  at $i\in S$ under $(\pi^{1},\pi^{2})$ by
\begin{align}\label{discre time total cost}
V_{\theta,D}(i,\pi^{1},\pi^{2}):=\limsup_{n\rightarrow \infty} \frac{1}{\theta n}ln\big(\mathbb{E}^{\pi^{1},\pi^{2}}_{i}\big[e^{\theta \sum_{k=0}^{n-1} D(X_{k},A_{k},B_{k})}\big]\big)
\end{align}
and $\theta$-optimal discrete time average value function, if it exists, is given by
\begin{align}\label{discre time optimal value function}
V^{*}_{\theta, D}(i):=\inf_{\pi^{1}}\sup_{\pi^{2}}V_{\theta,D}(i,\pi^{1},\pi^{2})=\sup_{\pi^{2}}\inf_{\pi^{1}}V_{\theta,D}(i,\pi^{1},\pi^{2})
\end{align}
It is easy to see that the value function $V^{*}_{\theta, D}(\cdot)$ satisfies the following.
\begin{align}\label{monotonicity}
V^{*}_{\theta, D}(\cdot)\leq V^{*}_{\theta, D_{1}}(\cdot) \hspace{.2cm} if \hspace{.2cm} D\leq D_{1}  \hspace{.2cm} and  \hspace{.2cm} V^{*}_{\theta, c+D}(\cdot)=c+V^{*}_{\theta, D}(\cdot)
\end{align}
where $c\in \mathbb{R}$. Since $D\leq D_{1}+\vert\vert D-D_{1}\vert \vert$ it follows that $V^{*}_{\theta, D}(\cdot)\leq V^{*}_{\theta, D_{1}}(\cdot)+\vert\vert D-D_{1}\vert \vert$. Similarly, by interchanging the roles of $D$ and $D_{1}$ this yields that 
\begin{align}\label{estimate}
\vert\vert V^{*}_{\theta, D}(\cdot)-V^{*}_{\theta, D_{1}}(\cdot)\vert\vert \leq \vert\vert D-D_{1}\vert \vert.
\end{align}
Observing that $V^{*}_{\theta, 0}=0$, the monotonicity property in (\ref{monotonicity}) yields that, for bounded continuous functions $D,D_{1}$,
\begin{align}
V^{*}_{\theta,D}\leq 0 \leq V^{*}_{\theta,D_{1}},  \hspace{.2cm} \textup{when} \hspace{.2cm} D\leq 0 \leq D_{1}.
\end{align} We have the following theorem.
\begin{theorem}\label{discrete}
	Under Assumptions 1 and 2, we have the following:
	\begin{itemize}
		\item[(i)] For each bounded continuous function $D$ on $\mathbb{K}$ there exist $\mu_{D}\in \mathbb{R}$ and $h_{D}:S\rightarrow \mathbb{R}$ such that
		\begin{align}\label{discrete optimality}
		e^{\theta[\mu_{D}+h_{D}(i)]}=&\sup_{\varphi\in \mathcal{P}(B(i))}\inf_{\psi\in \mathcal{P}(A(i))}\big[\int_{A(i)}\int_{B(i)}\psi(da)\varphi(db)
		e^{\theta D(i,a,b)}\sum_{j\in S}p_{i,j}(a,b)e^{\theta h(j)}\big]\nonumber \\&=\inf_{\psi\in \mathcal{P}(A(i))}\sup_{\varphi\in \mathcal{P}(B(i))}\big[\int_{A(i)}\int_{B(i)}\psi(da)\varphi(db)
		e^{\theta D(i,a,b)}\sum_{j\in S}p_{i,j}(a,b)e^{\theta h(j)}\big]\nonumber\\
		&\mbox{and}\quad\mu_{D}=V^*_{\theta, D}(i),   \hspace{.3cm} i\in S.
		\end{align}
		\item[(ii)] For bounded continuous functions $D,D_{1}$,
		\begin{align}
		\vert \mu_{D}-\mu_{D_{1}} \vert\leq \vert\vert D-D_{1} \vert\vert.
		\end{align}
	\end{itemize}
\end{theorem}
\begin{proof}
The proof of $(i)$ follows by putting together arguments and results from the existing literature on risk-sensitive control of discrete-time Markov chains. We just outline the steps.\\
\textbf{Step 1:} Using standard contraction argument it can be shown that for each $\beta \in (0,1)$ there exists function $V_{\beta}(\cdot)$ on $S$ satisfying
\begin{align}\label{discounted}
		e^{\theta V_{\beta}(i)}=&\sup_{\varphi\in \mathcal{P}(B(i))}\inf_{\psi\in \mathcal{P}(A(i))}\big[\int_{A(i)}\int_{B(i)}\psi(da)\varphi(db)
		e^{\theta D(i,a,b)}\sum_{j\in S}p_{i,j}(a,b)e^{\theta\beta V_{\beta}(j)}\big]\nonumber\\&=\inf_{\psi\in \mathcal{P}(A(i))}\sup_{\varphi\in \mathcal{P}(B(i))}\big[\int_{A(i)}\int_{B(i)}\psi(da)\varphi(db)
		e^{\theta D(i,a,b)}\sum_{j\in S}p_{i,j}(a,b)e^{\theta\beta V_{\beta}(j)}\big],   \hspace{.3cm} i\in S.
		\end{align} Also it is true that $||V_{\beta}||\leq \frac{||D||}{1-\beta}$.\\
\textbf{Step 2:} Fix a sequence $\beta_n\uparrow 1$. For $n\geq 1$, define
\begin{align}\label{normalisation}
z_{\beta_n}=\displaystyle\sup_{i\in S}V_{\beta_n}(i),\quad w_{\beta_n}(i)=V_{\beta_n}(i)-z_{\beta_n},\quad g_{\beta_n}=(1-\beta_n)z_{\beta_n}.
\end{align}	From \eqref{discounted} and \eqref{normalisation} we get 	
\begin{align}\label{normalised}
		e^{\theta w_{\beta_n}(i)+\theta g_{\beta_n}}=&\sup_{\varphi\in \mathcal{P}(B(i))}\inf_{\psi\in \mathcal{P}(A(i))}\big[\int_{A(i)}\int_{B(i)}\psi(da)\varphi(db)
		e^{\theta D(i,a,b)}\sum_{j\in S}p_{i,j}(a,b)e^{\theta\beta_n w_{\beta_n}(j)}\big]\\&=\inf_{\psi\in \mathcal{P}(A(i))}\sup_{\varphi\in \mathcal{P}(B(i))}\big[\int_{A(i)}\int_{B(i)}\psi(da)\varphi(db)
		e^{\theta D(i,a,b)}\sum_{j\in S}p_{i,j}(a,b)e^{\theta\beta_n w_{\beta_n}(j)}\big],   \hspace{.3cm} i\in S.
		\end{align}Now arguing as in Proposition 3.1 in \cite{Chen19}, it can be shown that there exists a subsequence of $\beta_n$, which we relabel as $\beta_n$ and function $h_D(i)$ and constant $\mu_D$ such that $h_D(i)=\lim_{n\to \infty}w_{\beta_n}(i)$ and $\mu_D=\lim_{n\to \infty}g_{\beta_n}$.\\
\textbf{Step 3:} Now taking limit in \eqref{normalised} and using Step 2 we get \eqref{discrete optimality}.\\
\textbf{Step 4:} The fact that $\mu_D=V^*_{\theta, D}(i)$ follows as in Lemma 2.3 in \cite{Daniel19}.\\
The proof of $(ii)$ is straightforward from part $(i)$ and \eqref{estimate}.

\end{proof}

\begin{lemma}\label{technical2}
	Suppose that Assumption 1 is valid and for each $g\in \mathbb{R}$ define the function $D_{g}:\mathbb{K}\rightarrow \mathbb{R}$ by
	\begin{align}\label{cost}
	D_{g}(i,a,b)=C(i,a,b)+\frac{1}{\theta}ln\big(\int_{0}^{B}e^{\theta[\int_{0}^{s}\rho_{(i,a,b)}(t)dt-gs]}dF_{i,a,b}(s)\big).
	\end{align}
	The following assertions hold.
	\begin{itemize}
		\item[(i)] $D_{g}$ is bounded continuous on $\mathbb{K}$ for each $g\in \mathbb{R}$. 
		\item[(ii)] $\vert\vert D_{g}-D_{g_{1}} \vert\vert  \leq B\vert g-g_{1}\vert$, $g,g_{1}\in \mathbb{R}$.
		\item[(iii)] There exist $g^{-}\geq 0$ such that $D_{g^{-}}\leq 0$.
		
		\item[(iv)] $D_{g_{+}}\geq 0$ for some $g^{+}\leq 0$.
	\end{itemize}
\end{lemma}
\begin{proof}
The proof is a straight forward generalization of Lemma 6.1 in \cite{Cadena16}.
\end{proof}
We finally have the existence theorem.
\begin{theorem}\label{existence}(\textbf{Existence of solutions})
	Under Assumptions 1 and 2, there exists $g\in \mathbb{R}$ and $h:S\rightarrow \mathbb{R}$ such that the optimality equation (\ref{optimal equation}) is satisfied.
\end{theorem}
\begin{proof}
For each $g$ consider $D_g$ given by \eqref{cost}. Combining Lemma \ref{technical2} and Theorem \ref{discrete} we get that $\mu_{D_g}$ is continuous in $g$. So again using Lemma \ref{technical2} and intermediate value property we get the existence of a $g$ such that $\mu_{D_g}=0$. Hence we have the result from Theorem \ref{discrete}.
\end{proof}

\section{Non-zero Sum Game Model}
In the non-zero sum game model we assume that there is no immediate cost and individual players have there own running cost functions. For $m=1,2$, we denote the running cost function for player $m$ by $\rho^m$. Here the evolution of the game is similar, except for the fact that upon taking their individual actions both players incur a holding cost upto the next transition. The definition of the policies is same as the zero-sum case. Thus, the total cost upto a positive time t for player 1 is given by:
\begin{align}\label{total cost for player 1}
\mathcal{C}_{t}^{1}=\sum_{k=0}^{N_{t}-1}\int_{0}^{S_{k}}\rho^1_{(X_{k},A_{k},B_{k})}(r)dr+\int_{0}^{t-T_{N_{t}}}\rho^1_{(X_{N_{t}},A_{N_{t}},B_{N_{t}})}(r)dr,
\end{align}
while for player 2 it is given by:
\begin{align}\label{total cost for player 2}
\mathcal{C}_{t}^{2}=\sum_{k=0}^{N_{t}-1}\int_{0}^{S_{k}}\rho^2_{(X_{k},A_{k},B_{k})}(r)dr+\int_{0}^{t-T_{N_{t}}}\rho^2_{(X_{N_{t}},A_{N_{t}},B_{N_{t}})}(r)dr.
\end{align}
Here the objective of each player is to minimise their own average costs.
\begin{definition} Fix a pair of policies $(\pi^1,\pi^2)\in \Pi_1\times \Pi_2$. 
	For $m=1,2$, define the value function for player 1 as $V_{\theta}^{1}(\pi^2)=\inf_{\pi^{1}}J_{\theta}^1(i,\pi^{1},\pi^{2})$, where $J_{\theta}^1$ is given by \eqref{theta sensitive cost}, with $\mathcal{C}_t$ replaced by $\mathcal{C}_t^1$. Similarly, the value function for player 2 is given by  $V_{\theta}^{2}(\pi^1)=\displaystyle\inf_{\pi^{2}}J_{\theta}^2(i,\pi^{1},\pi^{2})$, where $J_{\theta}^2$ is given by \eqref{theta sensitive cost}, with $\mathcal{C}_t$ replaced by $\mathcal{C}_t^2$.
\end{definition}
\begin{definition}(Nash equilibrium)
	A pair of policies $(\pi^{*^{1}},\pi^{*^{2}})\in \Pi_{1}\times \Pi_{2}$ is called a Nash equilibrium for the non-zero sum game if 
	\begin{align*}
	&J_{\theta}^1(i,\pi^{*^1},\pi^{*^2})\leq J_{\theta}^1(i,\pi^{1},\pi^{*^2}) \hspace{.3cm} and\\
	&J_{\theta}^2(i,\pi^{*^1},\pi^{*^2})\leq J_{\theta}^2(i,\pi^{*^1},\pi^{2}),
	\end{align*}
	for all $i\in S$ and  $(\pi^{1},\pi^{2})\in \Pi_{1}\times \Pi_{2}$.
\end{definition}

\section{Analysis of Non-Zero Sum Game}

We wish to establish the existence of Nash equilibrium for the non-zero sum game. To that end we, just like in the zero-sum case first consider a discrete time non-zero sum game given by the embedded Markov chain. Given two bounded continuous functions $D_1$ and $D_2$ on $\mathbb{K}$, we define for $i\in S$, under $(\pi^{1},\pi^{2})$, the discrete-time cost functional for player $m, m=1,2$ by
\begin{align}
V_{\theta,D_m}(i,\pi^{1},\pi^{2}):=\limsup_{n\rightarrow \infty} \frac{1}{\theta n}ln\big(\mathbb{E}^{\pi^{1},\pi^{2}}_{i}\big[e^{\theta \sum_{k=0}^{n-1} D_m(X_{k},A_{k},B_{k})}\big]\big).
\end{align}
We have the following discrete-time theorem.
\begin{theorem}\label{disc}
Suppose that Assumptions 1 and 2 are satisfied. Fix a pair of stationary strategies $(\phi^1,\phi^2)$. Then there exist functions $y^{\phi^1}$ and $y^{\phi^2}$ on $S$ and constants $\mu^{\phi^1}$ and $\mu^{\phi^2}$ such that the following are true. \\
	(i)\begin{align}\label{verification player 1}
	e^{\theta y^{\phi^2}(i)+\theta \mu^{\phi^2}}&=\inf_{\psi \in\mathcal{P}(A(i))}\biggl\{\int_{A(i)}\int_{B(i)}e^{\theta D_{1}(i,a,b)}\times \sum_{j\in S}e^{\theta y^{\phi^2}(j)}p_{i,j}(a,b)\phi^{2}(db\vert i)\psi(da)\biggr\}\,\,\forall i\in S,
	\end{align}
	and $\mu^{\phi^2}=\inf_{\pi^1}V_{\theta,D_1}(i,\pi^{1},\phi^{2})$.\\
	(ii)\begin{align}\label{verification player 2}
	e^{\theta y^{\phi^1}(i)+\theta \mu^{\phi^1}}&=\inf_{\varphi \in\mathcal{P}(B(i))}\biggl\{\int_{A(i)}\int_{B(i)}e^{\theta D_{2}(i,a,b)}\times \sum_{j\in S}e^{\theta y^{\phi^1}(j)}p_{i,j}(a,b)\varphi(db)\phi^{1}(da\vert i)\biggr\}\,\,\forall i\in S,
	\end{align}
	and $\mu^{\phi^1}=\inf_{\pi^2}V_{\theta,D_2}(i,\phi^{1},\pi^{2})$.
\end{theorem}
\begin{proof} The proof again follows by putting together arguments and result from the existing literature. So like in the zero-sum case we outline the steps.\\
\textbf{Step 1:} Let $\alpha\in (0,1)$. Then using a contraction argument the following can be shown.  
	 
	\begin{itemize}
		\item[(a)] For each fixed $\phi^{2} \in \Phi^{2}$, there exists a function $w^{\phi^{2},\alpha}$ such that
		\begin{align}\label{discounted approaches for player 1}
		e^{\theta w^{\phi^{2},\alpha}(i)}=\inf_{\psi \in\mathcal{P}(A(i))}\bigg[\int_{A(i)}\int_{B(i)} e^{\theta D_{1}(i,a,b)}\times \sum_{j\in S}e^{\theta \alpha w^{\phi^{2},\alpha}(j)}p_{i,j}(a,b) \phi^{2}(db\vert i)\psi(da)\bigg],
		\end{align}
		for all $i\in S$. 
        	\item[(b)] For each fixed $\phi^{1} \in \Phi^{1}$, there exists a function $w^{\phi^{1},\alpha}$ on $S$ such that
        \begin{align}\label{discounted approaches for player 2}
        e^{\theta w^{\phi^{1},\alpha}(i)}=\inf_{\varphi \in\mathcal{P}(B(i))}\bigg[\int_{A(i)}\int_{B(i)} e^{\theta D_{2}(i,a,b)}\times \sum_{j\in S} e^{\theta \alpha w^{\phi^{1},\alpha}(j)}p_{i,j}(a,b)\varphi(db) \phi^{1}(da\vert i)\bigg],
        \end{align}
        for all $i\in S$.   
\end{itemize}
\textbf{Step 2:}
Fix an arbitrary sequence $\{\alpha_{n}\}\in (0,1)$ satisfying $\alpha_{n}\uparrow 1$, as $n\rightarrow \infty$. For each $n\geq 1$ set
\begin{align*}
&\gamma_{\alpha_{n}}^{\phi^2}=\sup_{i\in S}w^{\phi^{2},\alpha_{n}}(i),   \hspace{1cm}  \gamma_{\alpha_{n}}^{\phi^1}=\sup_{i\in S}w^{\phi^{1},\alpha_{n}}(i),\\
&\mu^{\phi^2}_{\alpha_{n}}=(1-\alpha_{n})\gamma_{\alpha_{n}}^{\phi^2}, \hspace{1cm} \mu^{\phi^1}_{\alpha_{n}}=(1-\alpha_{n})\gamma_{\alpha_{n}}^{\phi^1}, \\
&v_{\alpha_{n}}^{\phi^2}(i)= w^{\phi^{2},\alpha_{n}}(i)-\gamma_{\alpha_{n}}^{\phi^2},    \hspace{1cm}  v_{\alpha_{n}}^{2}(i)= w^{\phi^{1},\alpha_{n}}(i)-\gamma_{\alpha_{n}}^{\phi^1}. 
\end{align*} Now arguing as in Proposition 3.1 in \cite{Chen19}, it can be shown that there exists functions $y^{\phi^1}$ and $y^{\phi^2}$ and constants $\mu^{\phi^1}$ and $\mu^{\phi^2}$ such that along a subsequence $y^{\phi^m}(i)=\lim_{n\rightarrow \infty}v^{\phi^m}_{\alpha_{n}}(i)$ and $\mu^{\phi^m}=\lim_{n\rightarrow \infty}\mu^{\phi^m}_{\alpha_{n}}$, for $m=1,2$.\\
\textbf{Step 3:} First we rewrite equations \eqref{discounted approaches for player 1} and \eqref{discounted approaches for player 2} in terms of the quantities defined in Step 2. Then taking limit $n\to \infty$ and using Step 2, we obtain equations \eqref{verification player 1} and \eqref{verification player 2} respectively.\\
\textbf{Step 4:} The interpretations of $\mu^{\phi^1}$ and $\mu^{\phi^2}$ follows by similar arguments as in Theorem 4.1 of \cite{Chen19}.
\end{proof}
In order to establish the existence of a Nash equilibrium we need the following additional assumption.
\begin{assumption} Fix a state $i^*\in S$. Define $\tau^*=\inf\{n\geq 1: X_n=i^*\}$. We assume that there exist constants $R$ and $M$ such that $$\sup_{\phi^1 \in \Phi^1}\sup_{\phi^2\in \Phi^2}\sup_{i\in S}\mathbb{E}_i^{\phi^1,\phi^2}\left[R^{\tau^*}\right]\leq M.$$ For this $R$, we further assume that $\theta$ is such that $$e^{2\theta B M_{\rho}}\leq R,$$ where $M_{\rho}=\max\{M_{\rho^1},M_{\rho^2}\}$ where $M_{\rho^i}$ is as in 
\eqref{finiteness of the cost rate} with $\rho$ replaced by $\rho^i$.
\end{assumption} For sufficient conditions ensuring the first part of Assumption 3, see Proposition 3 in \cite{Ghosh18}.
Next we obtain the following theorem as a consequence of the previous theorem.
\begin{theorem}\label{individual optimality} Assume that Assumptions 1, 2 and 3 hold. Fix $(\phi^1,\phi^2)\in \Phi^1\times \Phi^2$. Then 
	there exist constants $g^{\phi^1}, g^{\phi^2}$, real valued functions $h^{\phi^1},h^{\phi^2}$ on $S$ with $h^{\phi^1}(i^*)=h^{\phi^2}(i^*)=0$, such that the following are true.\\
	(i)\begin{align} 
	e^{\theta h^{\phi^2}(i)}&=\inf_{\psi \in\mathcal{P}(A(i))}\biggl\{\int_{A(i)}\int_{B(i)}\int_{0}^{B} e^{\theta[\int_{0}^{s}\rho^1_{(i,a,b)}(t)dt-g^{\phi^2}s]}dF_{i,a,b}(s)\nonumber\\&\times \sum_{j\in S}e^{\theta h^{\phi^2}(j)}p_{i,j}(a,b)\phi^{2}(db\vert i)\psi(da)\biggr\},\qquad\forall i\in S.
	\end{align}
	 
	(ii)\begin{align}
	e^{\theta h^{\phi^1}(i)}&=\inf_{\varphi \in\mathcal{P}(B(i))}\biggl\{\int_{A(i)}\int_{B(i)}\int_{0}^{B} e^{\theta[\int_{0}^{s}\rho^2_{(i,a,b)}(t)dt-g^{\phi^1}s]}dF_{i,a,b}(s)\nonumber\\&\times \sum_{j\in S}e^{\theta h^{\phi^1}(j)}p_{i,j}(a,b)\phi^{1}(da\vert i)\varphi(db)\biggr\},\qquad\forall i\in S.
	\end{align}
	(iii) $g^{\phi^2}=\inf_{\pi^1\in \Pi_1}J_{\theta}^1(i,\pi^1,\phi^2)$ for all $i$ and $g^{\phi^1}=\inf_{\pi^2\in \Pi_2}J_{\theta}^2(i,\phi^1,\pi^2)$ for all $i$.\\
	(iv) For $(i,a,b)\in \mathbb{K}$, let $D_1^{g^{\phi^2}}(i,a,b)=\frac{1}{\theta}ln \left(\int_0^B e^{\theta[\int_0^s\rho^1_{(i,a,b)}(t)dt-g^{\phi^2}s]}dF_{i,a,b}(s)\right)$ and\\ $D_2^{g^{\phi^1}}(i,a,b)=\frac{1}{\theta}ln \left(\int_0^B e^{\theta[\int_0^s\rho^2_{(i,a,b)}(t)dt-g^{\phi^1}s]}dF_{i,a,b}(s)\right)$. Then $h^{\phi^1}$ and $h^{\phi^2}$ have the following representations.
	$$h^{\phi^2}(i)=\inf_{\phi^1 \in \Phi^1}\frac{1}{\theta}ln \mathbb{E}_i^{\phi^1,\phi^2}\left[e^{\theta\displaystyle\sum_{k=0}^{\tau^*-1} D_1^{g^{\phi^2}}(X_k,A_k,B_k)}\right],\,\,\forall i \in S\setminus \{i^*\}.$$
	$$h^{\phi^1}(i)=\inf_{\phi^2 \in \Phi^2}\frac{1}{\theta}ln \mathbb{E}_i^{\phi^1,\phi^2}\left[e^{\theta\displaystyle\sum_{k=0}^{\tau^*-1} D_2^{g^{\phi^1}}(X_k,A_k,B_k)}\right],\,\,\forall i \in S\setminus \{i^*\}.$$
\end{theorem}	
\begin{proof}
The proof of (i) and (ii) follows from Theorem \ref{disc} by a similar trick as in Theorem \ref{existence} of the zero-sum game section. Proof of (iii) follows by arguments similar to Theorem \ref{verification theorem}. Finally, the proof of (iv) follows by arguments similar to Lemma 8.1 in \cite{Chen19}.
\end{proof}
Now, fix any $(\phi^{1},\phi^{2})\in \Phi^{1}\times \Phi^{2}$. Define
\begin{align*}
\Delta(\phi^{2})&=\biggl\{\phi^{*^1}\in \Phi^{1}: for \hspace{.1cm} each \hspace{.1cm} i\in S,\hspace{.1cm} \int_{A(i)}\int_{B(i)}\int_{0}^{B} e^{\theta[\int_{0}^{s}\rho^1_{(i,a,b)}(t)dt-g^{\phi^2}s]}dF_{i,a,b}(s)\\&\times \sum_{j\in S}e^{\theta h^{\phi^2}(j)}p_{i,j}(a,b)\phi^{2}(db\vert i)\phi^{*^1}(da\vert i)= \inf_{\psi \in\mathcal{P}(A(i))}\biggl\{\int_{A(i)}\int_{B(i)} \int_{0}^{B} e^{\theta[\int_{0}^{s}\rho^1_{(i,a,b)}(t)dt-g^{\phi^2}s]}dF_{i,a,b}(s)\\& \times \sum_{j\in S}e^{\theta h^{\phi^2}(j)}p_{i,j}(a,b)\phi^{2}(db\vert i)\psi(da)\biggr\}\bigg\}
\end{align*}
and 
\begin{align*}
\Delta(\phi^{1})&=\biggl\{\phi^{*^2}\in \Phi^{2}: for \hspace{.1cm} each \hspace{.1cm} i\in S,\hspace{.1cm} \int_{A(i)}\int_{B(i)}\int_{0}^{B} e^{\theta[\int_{0}^{s}\rho^2_{(i,a,b)}(t)dt-g^{\phi^1}s]}dF_{i,a,b}(s)\\&\times \sum_{j\in S}e^{\theta h^{\phi^1}(j)}p_{i,j}(a,b)\phi^{*^2}(db\vert i)\phi^{1}(da\vert i)= \inf_{\varphi \in\mathcal{P}(B(i))}\biggl\{\int_{A(i)}\int_{B(i)}\int_{0}^{B} e^{\theta[\int_{0}^{s}\rho^2_{(i,a,b)}(t)dt-g^{\phi^1}s]}dF_{i,a,b}(s)\\&\times \sum_{j\in S}e^{\theta h^{\phi^1}(j)}p_{i,j}(a,b)\varphi(db)\phi^{1}(da\vert i)\biggr\}\bigg\}.
\end{align*}
It follows from our assumptions that the sets $\Delta(\phi^{2})$ and $\Delta(\phi^{1})$ are non-empty.
\begin{lemma} Suppose that Assumptions 1 and 2 is true.
	For each $(\phi^{1},\phi^{2})\in \Phi^{1}\times \Phi^{2}$, $\Delta(\phi^{2})\times \Delta(\phi^{1})$ is convex and compact with respect to the weak topology.
\end{lemma}
\textbf{Proof.} We first show that $\Delta(\phi^{2})$ is convex. For that let $\tilde{\phi}^{1},\tilde{\psi}^{1}\in \Delta(\phi^{2})$ and $\lambda\in [0,1]$, define: $\phi^{1}_{\beta}(\cdot|i)=\lambda \tilde{\phi}^{1}(\cdot\vert i)+(1-\lambda)\tilde{\psi}^{1}(.\vert i)$ for all $i\in S$. By writing down the expression of $\phi^{1}_{\beta}$ one easily gets that $\phi^{1}_{\beta}\in \Delta(\phi^{2})$. Thus $\Delta(\phi^{2})$ is convex. By analogous argument $\Delta(\phi^{1})$ is also convex, which together implies that $\Delta(\phi^{2})\times \Delta(\phi^{1})$ is convex.\\
By the compactness of $\Phi^{1}\times \Phi^{2}$ and the fact that $\Delta(\phi^{2})\times \Delta(\phi^{1})$ is a subset of $\Phi^{1}\times \Phi^{2}$, its enough to show that $\Delta(\phi^{2})\times \Delta(\phi^{1})$ is a closed subset. First we show that $\Delta(\phi^{2})$ is a closed subset of the compact space $\Phi^{1}$. Let $\{\phi^{*^1}_{n}\}\subset \Delta(\phi^{2})$ be an arbitrary sequence converging to $\phi^{*^1}\in \Phi^{1}$, and $G(i,a):=\int_{B(i)}\int_{0}^{B} e^{\theta[\int_{0}^{s}\rho^1_{(i,a,b)}(t)dt-g^{\phi^2}s]}dF_{i,a,b}(s)\times \sum_{j\in S}e^{\theta h^{\phi^2}(j)}p_{i,j}(a,b)\phi^{2}(db\vert i)$ for $i\in S$ and $a\in A(i)$. By Assumption 1, we have that for each $i\in S$, $G(i,.)$ is a bounded continuous function on $A(i)$. Thus by definition of weak topology we obtain 
\begin{align*}
\int_{A(i)}G(i,a)\phi^{*^1}_{n}(da\vert i)\rightarrow \int_{A(i)}G(i,a)\phi^{*^1}(da\vert i).
\end{align*}
as $n\rightarrow \infty$. Since $\{\phi^{*^1}_{n}\}\subset \Delta(\phi^{2})$  
\begin{align*}
\int_{A(i)}G(i,a)\phi^{*^1}_{n}(da\vert i)= \inf_{\mu \in \mathcal{P}(A(i))}\biggl\{\int_{A(i)}G(i,a)\mu(da)\biggr\}
\end{align*}
for all $n=1,2,...$. Hence, we have $\phi^{*^1}\in \Delta(\phi^{2})$. Thus, $\Delta(\phi^{2})$ is closed. Similarly, $\Delta(\phi^{1})$ is closed. So combining we get $\Delta(\phi^{2})\times \Delta(\phi^{1})$ is convex and compact.
\begin{lemma}\label{continuity} Suppose that Assumptions 1,2 and 3 hold. For each $i\in S$, the functions $\phi^1\to h^{\phi^1}(i)$ and $\phi^2\to h^{\phi^2}(i)$ are continuous in $\phi^1\in \Phi^1$ and $\phi^2\in \Phi^2$ respectively. Continuity also holds for the functions $\phi^1\to g^{\phi^1}$ and $\phi^2\to g^{\phi^2}$.
\end{lemma}
\begin{proof}
By (iii) of Theorem \ref{individual optimality}, we have $|g^{\phi^1}|\leq M_{\rho}$ and $|g^{\phi^1}|\leq M_{\rho}$. We also have $||D^{\phi^2}_1||\leq 2BM_{\rho}$ and $||D^{\phi^1}_2||\leq 2BM_{\rho}$. Thus by Assumption 3, we have $h^{\phi^m}(i)\leq \frac{1}{\theta}ln M$ for $m=1,2$ and for all $i\in S$. Now Assumption 3 also implies that $\sup_{\pi^1\in \Phi^1}\sup_{\phi^2\in \Phi^2}\sup_{i\in S}\mathbb{E}_i^{\phi^1,\phi^2}\tau^*\leq K$, for some $K$. So by Jensen's inequality we have $h^{\phi^m}(i)\geq -2KBM_{\rho}$ for $m=1,2$ and for all $i\in S$. Now suppose $\phi^2_n\to \phi^2$. Let us consider subsequences $\{g^{\phi^2_{n_k}}\}$, $\{h^{\phi^2_{n_k}}(i)\}$ . We will get a further subsequence such that $g^{\phi^2_{n_l}}\to g^*$ for some constant $g^*$ and $h^{\phi^2_{n_l}}(j)\to u(j)$ for all $j\in S$ for some function $u$ on $S$. We have,
\begin{align} 
	e^{\theta h^{\phi^2_{n_l}}(i)}&=\inf_{\psi \in\mathcal{P}(A(i))}\biggl\{\int_{A(i)}\int_{B(i)}\int_{0}^{B} e^{\theta[\int_{0}^{s}\rho^1_{(i,a,b)}(t)dt-g^{\phi^2_{n_l}}s]}dF_{i,a,b}(s)\nonumber\\&\times \sum_{j\in S}e^{\theta h^{\phi^2_{n_l}}(j)}p_{i,j}(a,b)\phi^{2}_{n_l}(db\vert i)\psi(da)\biggr\}\qquad\forall i\in S.
	\end{align}
	Now by our assumptions, definition of weak convergence and extended Fatou's lemma (Lemma 8.3.7 in \cite{Lasserre99}), we obtain by taking limit $l\to \infty$ in the above equation,
	\begin{align} 
	e^{\theta u(i)}&=\inf_{\psi \in\mathcal{P}(A(i))}\biggl\{\int_{A(i)}\int_{B(i)}\int_{0}^{B} e^{\theta[\int_{0}^{s}\rho^1_{(i,a,b)}(t)dt-g^*s]}dF_{i,a,b}(s)\nonumber\\&\times \sum_{j\in S}e^{\theta u(j)}p_{i,j}(a,b)\phi^{2}(db\vert i)\psi(da)\biggr\}\qquad\forall i\in S.
	\end{align} Thus again arguing as in Theorem \ref{individual optimality}, we will get that $g^*=\inf_{\pi^1 \in \Pi_1}J_{\theta}^1(i,\pi^1,\phi^2)=g^{\phi^2}$ and $u(i)=\inf_{\phi^1 \in \Phi^1}\frac{1}{\theta}ln \mathbb{E}_i^{\phi^1,\phi^2}\left[e^{\theta\displaystyle\sum_{k=0}^{\tau^*-1} D_1^{g^{\phi^2}}(X_k,A_k,B_k)}\right]=h^{\phi^2}(i)\,\,\forall i \in S\setminus \{i^*\}$. Since every subsequence has a further subsequence which converges to the same limit, we are done.
\end{proof}
Now we state the main theorem of this section.
\begin{theorem} Suppose that Assumptions 1,2 and 3 hold.
	There exists constants $g^{*^1}, g^{*^2}$, real valued functions $y^{*^1},y^{*^2}$ on S and a pair of stationary policies  $(\phi^{*^1},\phi^{*^{2}})\in \Phi^{1}\times \Phi^{2}$ such that
	{\footnotesize\begin{align}\label{couple1}
	e^{\theta y^{*^1}(i)}&=\inf_{\psi \in\mathcal{P}(A(i))}\biggl\{\int_{A(i)}\int_{B(i)}\int_{0}^{B} e^{\theta[\int_{0}^{s}\rho^1_{(i,a,b)}(t)dt-g^{*^1}s]}dF_{i,a,b}(s)\times \sum_{j\in S}e^{\theta y^{*^1}(j)}p_{i,j}(a,b)\phi^{*^{2}}(db\vert i)\psi(da)\biggr\}\nonumber\\
	&=\int_{A(i)}\int_{B(i)}\int_{0}^{B} e^{\theta[\int_{0}^{s}\rho^1_{(i,a,b)}(t)dt-g^{*^1}s]}dF_{i,a,b}(s)\times \sum_{j\in S}e^{\theta y^{*^1}(j)}p_{i,j}(a,b)\phi^{*^{2}}(db\vert i)\phi^{*^1}(da\vert i),
	\end{align}}
	and 
	{\footnotesize\begin{align}\label{couple2}
	e^{\theta y^{*^2}(i)}&=\inf_{\varphi \in\mathcal{P}(B(i))}\biggl\{\int_{A(i)}\int_{B(i)}\int_{0}^{B} e^{\theta[\int_{0}^{s}\rho^2_{(i,a,b)}(t)dt-g^{*^2}s]}dF_{i,a,b}(s)\times \sum_{j\in S}e^{\theta y^{*^2}(j)}p_{i,j}(a,b)\varphi(db)\phi^{*^1}(da\vert i)\biggr\}\nonumber\\
	&=\int_{A(i)}\int_{B(i)}\int_{0}^{B} e^{\theta[\int_{0}^{s}\rho_{(i,a,b)}(t)dt-g^{*^2}s]}dF_{i,a,b}(s)\times \sum_{j\in S}e^{\theta y^{*^2}(j)}p_{i,j}(a,b)\phi^{*^2}(db\vert i)\phi^{*^1}(da\vert i),
	\end{align}}
	for all $i\in S$. Moreover, the pair of policies $(\phi^{*^1},\phi^{*^{2}})\in \Phi^{1}\times \Phi^{2}$ is a Nash-equilibrium and we have $J_{m}(i,\phi^{*^1},\phi^{*^{2}})=g^{*^m}$ for all $i\in S$ and $m=1,2$.
\end{theorem}
\begin{proof} Let $2^{\Phi^{1}\times \Phi^{2}}$ be the power set of $\Phi^{1}\times \Phi^{2}$ and define the multi function $\Psi:\Phi^{1}\times \Phi^{2} \rightarrow 2^{\Phi^{1}\times \Phi^{2}}$ by $\Psi((\phi^{1}, \phi^{2}))=\Delta(\phi^{2})\times \Delta(\phi^{1})$. Next we show that $\Psi$ has a closed graph. Let $\{(\phi_{n}^{1},\phi_{n}^{2})\}\subset \Phi^{1}\times \Phi^{2}$ and $\{(\phi_{n}^{*^1},\phi_{n}^{*^2})\}\subset \Phi^{1}\times \Phi^{2}$ be arbitrary sequences with $\{(\phi_{n}^{*^1},\phi_{n}^{*^2})\}\in  \Psi((\phi^{1}_{n},\phi_{n}^{2}))$ and $\{(\phi_{n}^{1},\phi_{n}^{2})\}$ and $\{(\phi_{n}^{*^1},\phi_{n}^{*^2})\}$ converges to $(\bar{\phi}^{1},\bar{\phi}^{2})$ and $(\bar{\phi}^{*^1},\bar{\phi}^{*^2})$, respectively. Then by the definition of $\Delta(\phi^{2}_{n})$, we have 
\begin{align}\label{definition of delta 2}
&\inf_{\psi \in\mathcal{P}(A(i))}\biggl\{\int_{A(i)}\int_{B(i)}\int_{0}^{B} e^{\theta[\int_{0}^{s}\rho^1_{(i,a,b)}(t)dt-g^{\phi_n^2}s]}dF_{i,a,b}(s)\times \sum_{j\in S}e^{\theta h^{\phi^2_n}(j)}p_{i,j}(a,b)\phi^{2}_{n}(db\vert i)\psi(da)\biggr\}\nonumber\\
&=\int_{A(i)}\int_{B(i)}\int_{0}^{B} e^{\theta[\int_{0}^{s}\rho^1_{(i,a,b)}(t)dt-g^{\phi^2_n}s]}dF_{i,a,b}(s)\times \sum_{j\in S}e^{\theta h^{\phi^2_n}(j)}p_{i,j}(a,b)\phi^{2}_{n}(db\vert i)\phi^{*^1}_{n}(da\vert i).
\end{align}
Now using our assumptions, Lemma \ref{continuity} and extended Fatou's lemma (Lemma 8.3.7 in \cite{Lasserre99}) we obtain by taking limit $n\to \infty$ in \eqref{definition of delta 2},
\begin{align*}
&\inf_{\psi \in\mathcal{P}(A(i))}\biggl\{\int_{A(i)}\int_{B(i)}\int_{0}^{B} e^{\theta[\int_{0}^{s}\rho^1_{(i,a,b)}(t)dt-g^{\bar{\phi}^2}s]}dF_{i,a,b}(s)\times \sum_{j\in S}e^{\theta h^{\bar{\phi}^2}(j)}p_{i,j}(a,b)\bar{\phi}^{2}(db\vert i)\psi(da)\biggr\}\nonumber\\
&=\int_{A(i)}\int_{B(i)}\int_{0}^{B} e^{\theta[\int_{0}^{s}\rho^1_{(i,a,b)}(t)dt-g^{\bar{\phi}^2}s]}dF_{i,a,b}(s)\times \sum_{j\in S}e^{\theta h^{\bar{\phi}^2}(j)}p_{i,j}(a,b)\bar{\phi}^{2}(db\vert i)\bar{\phi}^{*^1}(da\vert i).
\end{align*}
for all $i\in S$, which implies $\bar{\phi}^{*^1}\in \Delta(\bar{\phi}^{2})$. Using similar arguments as above, we can also show that $\bar{\phi}^{*^2}\in \Delta(\bar{\phi}^{1})$. Hence, the multi function $\Psi$ has a closed graph. Therefore by Fan's fixed point theorem \cite{Fan52} we have the existence of $(\phi^{*^1},\phi^{*^2})\in \Phi^{1}\times \Phi^{2}$ such that $((\phi^{*^1},\phi^{*^2}))\in \Delta(\phi^{*^2})\times \Delta(\phi^{*^1})$. Now using Theorem \ref{individual optimality} we obtain solution to the coupled system of equations \eqref{couple1} and \eqref{couple2}.

Now for the Nash equilibrium part, it follows from \eqref{couple1} and arguments similar to Theorem \ref{verification theorem}, that $$g^{*^1}=J_1(i,\phi^{*^1},\phi^{*^2})=V_{\theta}^{1}(\phi^{*^2}).$$ Analogously, starting from \eqref{couple2} it can be shown that $$g^{*^2}=J_2(i,\phi^{*^1},\phi^{*^2})=V_{\theta}^{2}(\phi^{*^1}).$$ Hence we are done.
\end{proof}
\section{Conclusion} In this paper we have studied both zero-sum and non-zero sum risk-sensitive average criterion games for semi-Markov process. Here we assume that the state space is finite and the sojourn time distributions are supported within a fixed compact interval. So it remains an open problem to extend the setting to more general state space and sojourn time distributions. Note that such a problem is also open for the control case as well, because in \cite{Cadena16} where the control problem is studied similar assumptions are made and crucially used in the analysis.

\textbf{Acknowledgement:} The research of the second named author is supported by the Mathematical Research Impact Centric Support (MATRICS) grant, File No: MTR/2020/000350, by the Science and Engineering Research Board (SERB), Department of Science and Technology (DST), Government of India.

\bibliographystyle{plain}
	\bibliography{bib}
	
\end{document}